\documentclass{amsart}
\usepackage{amsmath}
\usepackage[colorinlistoftodos]{todonotes}
\usepackage{amsthm}
\usepackage{amssymb}

\usepackage{tikz}
\usetikzlibrary{positioning}

\newtheorem{theorem}{Theorem}[section]
\newtheorem{corollary}[theorem]{Corollary}
\newtheorem{lemma}[theorem]{Lemma}

\newtheorem{remark}[theorem]{Remark}

\newtheorem{example}{Example}

\usepackage[style = numeric, sorting = none]{biblatex}
\newcommand{\Q}{\mathbb{Q}}
\newcommand{\eps}{\epsilon}
\newcommand{\F}{\mathbb{F}}
\newcommand{\Z}{\mathbb{Z}}
\newcommand{\calC}{C}
\newcommand{\Hol}{\mathrm{Hol}}
\newcommand{\Aut}{\mathrm{Aut}}

\title{The Galois group of $x^{2p} + bx^p + c^p$ over $\Q$}
\author{Akash Jim}
\address{Princeton University\\
304 Washington Road \\
Princeton, NJ 08540}
\email{ajim@princeton.edu}
\author{Thomas Hagedorn}
\address{Department of Mathematics and Statistics\\
The College of New Jersey\\
2000 Pennington Road\\
Ewing, NJ 08618}
\email{hagedorn@tcnj.edu}
\date{\today}

\subjclass[2020]{12F10, 12E05, 11R09}
\keywords{Galois group, irreducible, trinomial, Dickson polynomial, power compositional}
\begin{document}

\maketitle

\begin{abstract}
    We prove an irreducibility criterion for polynomials of the form $h(x)=x^{2m} + bx^m + c_1 \in F[x]$ relating to the Dickson polynomials of the first kind $D_p$. In the case when $F = \mathbb{Q}$, $m$ is a prime $p>3$, and $c_1=c^p$, for $c\in\Q$, we explicitly determine the Galois group of $d_h= D_p(x, c) + b$, which is $\mathrm{Aff}(\mathbb{F}_p)$ or $C_p \rtimes C_{(p - 1)/2} \vartriangleleft \mathrm{Aff}(\mathbb{F}_p)$, and the Galois group of $h$, which is $C_2 \times \mathrm{Aff}(\mathbb{F}_p), \mathrm{Aff}(\mathbb{F}_p)$, or $C_2 \times (C_p \rtimes C_{(p - 1)/2}) \vartriangleleft C_2 \times \mathrm{Aff}(\mathbb{F}_p)$.
\end{abstract}

\section{Introduction}
Let $F$ be a field, $f(x)\in F[x]$ a polynomial, and let $K/F$ be the splitting field of $f(x)$ over $F$.  If $f(x)$ is a separable polynomial, the central result of Galois theory is that there is a bijective correspondence between the subgroups of $G=\mathrm{Gal}(K/F)$, the automorphisms of $K$ that fix $F$, and the subfields $L\subset K$ containing $F$. The group $G$ contains much information about the field extension $K/F$.  In particular, Galois proved (see \cite[VI, Thm. 7.2]{lang}) that the roots of $f(x)$ can be found via the usual arithmetic operations and $n$th roots precisely when $G$ is a solvable group.

When $F=\Q$, the Galois group of $f(x)$ has been determined for a number of classes of polynomials of small degree.  
For example, the Galois group of the trinomial $f(x)=x^{2k} + bx^k + c \in \mathbb{Q}[x]$ has been determined in the case of some small $k$.  The case $k=3$ was addressed in the work of \cite{sextics} and \cite{power-compositional}, and the case $k=4$  was solved by \cite{octics}, in the case when $c=1$ or a rational square.  In \cite{reciprocal}, Jones solves the case when $k$ is a prime $p>3$, $c=1$, and $b$ is an integer with $|b| \geq 3$.  In this paper, we expand on these results and determine the Galois group of the trinomial polynomial $f(x)$ in the case when $k>3$ is prime and $c$ is a $p$-th power in $\Q$.  In the process of determining this Galois group, we are also able to determine the Galois group of a related family of polynomials formed from the Dickson polynomials.

In \cite{dickson}, Dickson introduced a set of polynomials that often give automorphisms of 
the finite field $\F_{q^r}$.  The Dickson polynomials of the first kind are defined by 
\begin{equation*}
    \begin{aligned}
       D_1(t, n) &= t\\ 
       D_2(t, n) &= t^2 - 2n\\
       D_{k}(t, n) &= tD_{k -1}(t, n) - nD_{k-2}(t, n)\text{ for }k>2.  
    \end{aligned}
\end{equation*}
$D_k(t,n)$ is a degree $k$ polynomial and the next few polynomials are:
\begin{equation*}
    \begin{aligned}
        D_3(t, n)&= t^3-3nt\\
        D_4(t,n)&=t^4-4nt^2+2n^2\\
        D_5(t,n)&=t^5-5nt^3+5n^2t\\
        D_6(t,n)&=t^6-6nt^4+9n^2t^2-2n^3\\
        D_7(t,n)&=t^7-7nt^5+14n^2t^3-7n^3t\\
    \end{aligned}
\end{equation*}
A  closed-form expression for $D_k(t,n)$ is given by:
\begin{equation*}
    {\displaystyle D_{k}(t,n )=\sum _{i=0}^{\left\lfloor {\frac {k}{2}}\right\rfloor }{\frac {k}{k-i}}{\binom {k-i}{i}}(-n )^{i}t^{k-2i}\,.}
\end{equation*}

We first establish a reducibility criterion that relates to the Dickson polynomials. It generalizes Theorem 1.1(1) of Jones \cite{reciprocal}, who considered the case when $h \in \Z[x]$, $m$ is odd, $c= 1$, and $|b| \geq 3$.

\begin{theorem} \label{reducibility-criterion}
    Let $F$ be a field and $m>1$. The polynomial $h(x) = x^{2m} + bx^m + c \in F[x]$ is reducible if and only if one of the following conditions holds:
    \begin{enumerate}
        \item $f(x) = x^2 + bx + c$ is reducible; or
        \item\label{red-item2} For some prime $p$ divisor of $m$, there exist $n,\, t \in F$ with $c = n^p$ and $b = -D_p(t, n)$; or
        \item\label{red-item3} $4 \mid m$ and there exist $n, t \in F$ with $c = 16n^4$ and $b = 4D_4(t, n)$.
    \end{enumerate} 
\end{theorem}

\begin{remark}
    If $F=\Q$, Theorem~\ref{reducibility-criterion} allows one to determine the reducibility of $x^{2m} + bx^m + c \in \Q[x]$ by using only the Rational Root Test to test for a rational root of the polynomials $d_h(x)= D_p(x,c)+b$ in \eqref{red-item2} and $D_4(x, n)-b$ in \eqref{red-item3}.
\end{remark}

When $p$ is a prime and $h(x) = x^{2p} + bx^p + c^p \in \mathbb{Q}[x]$, we will see that the Galois group of $h$ is closely related to the Galois group of the polynomial
\begin{equation*}
 d_h(x)= d_{p, b, c}(x) =   D_p(x, c)+b,
\end{equation*}
where $p$ is a prime and $b, c\in F$.  
When $h$ is irreducible, we prove in Theorem \ref{galois-groups} that $d_h(x)$ is irreducible and we explicitly classify the Galois groups of $h$ and $d_h$. 
Theorem \ref{galois-groups} generalizes Theorem 1.1(2) of Jones
\cite{reciprocal}, who considered the case when $h\in \Z[x]$, $c=1$, and $|b| \geq 3$.

\begin{theorem}\label{galois-groups}
    Let $p>3$ be prime and assume $h(x) := x^{2p} + bx^p + c^p \in \mathbb{Q}[x]$ is irreducible. Then $d_h$ is irreducible and the Galois groups of $d_{h}$ and $h$ are determined as follows:
    \begin{enumerate}
        \item\label{frob-item1} If $b^2 - 4c^p \notin (-1)^{p(p-1)/2}p\mathbb{Q}^2$, the Galois group of $d_h$ is the affine group of $\mathbb{F}_p$, $\mathrm{Aff}(\mathbb{F}_p) \simeq C_p \rtimes C_{p - 1}$, and the Galois group of $h$ is $\mathrm{Aff}(\mathbb{F}_p) \times C_2$.
        \item\label{frob-item2} If $p \equiv 1 \bmod 4$ and $b^2 - 4c^p \in p\mathbb{Q}^2$, the Galois groups of $d_h$ and $h$ are identically $\mathrm{Aff}(\mathbb{F}_p)$.
        \item\label{frob-item3} If $p \equiv 3 \bmod 4$ and $b^2 - 4c^p \in -p\mathbb{Q}^2$, then the Galois group of $d_h$ is $C_p \rtimes C_{(p - 1)/2} \vartriangleleft \mathrm{Aff}(\mathbb{F}_p)$, and the Galois group of $h$ is $(C_p \rtimes C_{(p - 1)/2}) \times C_2$.
    \end{enumerate}
\end{theorem}

\begin{remark}
   In \cite{reciprocal}, the Galois group of $d_h$ in Theorem~\ref{galois-groups} (when $c=1$) is described using $\Hol(\calC_n)$, the Holomorph of $\calC_n$. It is defined as the semi-direct product: 
   \begin{equation*}
   \Hol(C_n):= C_n\rtimes \Aut(\calC_n)
   \end{equation*}
   Then the Galois group of $h$ is $\Hol(\calC_{2p})$ in part~\eqref{frob-item1} and $\Hol(\calC_p)$ in 
   part~\eqref{frob-item2}.  Part~\eqref{frob-item3} does not appear as part of  \cite[Thm 1.1(2)]{reciprocal} as $b^2-4>0$ by its assumption that $|b|\geq 3$ but the hypothesis of part~\eqref{frob-item3} is that $b^2-4\leq 0$. 
\end{remark}
\begin{example}
Let $p = 5$, and consider $h(x) = x^{10} - 3x^5 + 32$. By Theorem~\ref{reducibility-criterion}, $h$ is irreducible. By Theorem~\ref{galois-groups}, $h$ has Galois group $\mathrm{Aff}(\mathbb{F}_5) \times C_2$ and
\begin{equation*}
    d_h(x) = x^5-10x^3+20x-3
\end{equation*}
has Galois group $\mathrm{Aff}(\mathbb{F}_5)$.
\end{example}

\begin{remark}
    The case of $p = 2$ is a sub-case of the quartic, which has long been solved. Here, the Galois group of $h$ is $C_4$, and the Galois group of $d_h$ is $C_2$.
\end{remark}
    
\begin{remark}
    The case of $p = 3$ has been solved by Awtrey, Buerle, and Griesbach using resolvents as Example 4.1 of \cite{power-compositional}, over a general field. This generalizes Harrington and Jones's \cite{sextics}. In the case when $p=3$, the statement of Theorem \ref{galois-groups} is true, and can be derived from the work of \cite{power-compositional}. (The Galois group must either be $C_6 \simeq (C_3 \rtimes C_1) \times C_2$ or $D_6 \simeq S_3 \times C_2 \simeq \mathrm{Aff}(\mathbb{F}_3) \times C_2$, which can easily be distinguished by the degree of the extension.) For technical reasons, our proof of  Theorem~\ref{galois-groups} does not work when $p=3$.
\end{remark}

\begin{remark}\label{Waring}
    The Dickson polynomials of the first kind appear here because of a special case of Waring's identity, \cite{dickson-lidl}: $$\beta^k + \beta_1^k = D_k(\beta + {\beta_1}, \beta{\beta_1})$$ The case of $n = 1$, with an arbitrary $k$, is known as the $k$th Vieta-Lucas polynomial of $t$ \cite{vieta}. These polynomials are used in the work of Jones \cite{reciprocal} in the case of polynomials $x^{2m} + Ax^m + 1 \in \mathbb{Z}[x]$. 
\end{remark}

\section{Some Preliminaries}

Let $F$ denote an arbitrary field. $\mathrm{N}^K_F(\alpha)$ and $\mathrm{Tr}^K_F(\alpha)$ will denote the norm and trace of $\alpha \in K$ with respect to $F$. Let $f(x)=f_{b,c}(x)$ denote the monic trinomial $x^2 + bx + c \in F[x]$, and denote its roots by $\alpha, \alpha_1$ and its determinant by $\Delta := b^2 - 4c$.

A key theorem in establishing irreducibility of power compositional polynomials (cf., e.g., \cite{sextics}, \cite{octics}) is Capelli's Theorem (see Theorem 22 of \cite{schinzel} for proof):

\begin{theorem}[Capelli's Theorem] \label{Capelli}
Let  $f(x), g(x) \in F[x]$ and assume $f(x)$ is irreducible with root $\alpha$.  Then $f(g(x))$ is reducible in $F[x]$ if and only if the polynomial $g(x) - \alpha$ is reducible in $F(\alpha)[x]$. Moreover, if $g(x) - \alpha $ has the prime factorization $
g(x) - \alpha= C\prod_{i = 1}^r h_i(x)^{e_i}$, with $C\in F(\alpha)$ and irreducible polynomials $h_i(x)\in F(\alpha)[x]$, then 
\begin{equation*}
  f(g(x)) = \tilde{C}\prod_{i = 1}^r \mathrm{N}_F^{F(\alpha)}(h_i(x))^{e_i},  
\end{equation*} where $\tilde{C} \in F$ and
$\mathrm{N}_F^{F(\alpha)}(h_i(x))$ are irreducible. 
\end{theorem}

Two well known consequences of
this theorem are:

\begin{theorem}[Capelli] \label{cyclotomic} \cite[VI, Thm. 9.1]{lang}
    Let $a \in F, n \in \mathbb{N}$. $x^n - a$ is irreducible in $F[x]$ if and only if the following conditions hold:
    \begin{enumerate}
    \item $a \notin F^p $ for all primes $p\mid n$; and
    \item If $4 \mid n$, $a \notin  -4F^4$. 
    \end{enumerate}
\end{theorem}

\begin{corollary} \label{capelli-corollary}
    For $p$ prime, $x^p - \alpha \in F[x]$ is reducible if and only if $\exists \beta \in F$ with $\alpha = \beta^p$.
\end{corollary}

Thus, if $f(x^p)$ is reducible, either $f(x)$ is reducible, or there exists $\beta \in F(\alpha)$ with $f(\beta^p) = 0$ and $\mathrm{N}^{F(\alpha)}_{F}(x - \beta) \mid h(x)$. 

Define $p(x) := h(x) / \mathrm{N}(x - \beta)$.  The polynomial $p(x)$ is symmetric in $\beta$ and its conjugate $\mathrm{N}(\beta)/\beta \in F(\alpha)$, so it can be expressed in terms of the symmetric polynomials. By expanding and equating coefficients, we will obtain the reducibility criterion of \ref{reducibility-criterion}.

\section{A Reducibility Criterion for $x^{2m} + bx^m + c \in F[x]$}

In this section, we prove Theorem~\ref{reducibility-criterion}, which gives a
reducibility criteria for $x^{2m}+bx^m+c\in F[x]$. Theorem~\ref{reducibility-criterion} extends Jones's reducibility criterion \cite[Thm. 1.1(1)]{reciprocal}, which considered the case when $c=1$.  The proof of 
Theorem~\ref{reducibility-criterion} is essentially the same as Jones' proof in the case $c=1$.
It will follow as a corollary from 
Theorem~\ref{power-criterion}, which gives a relationship of the 
coefficients of 
\begin{equation*}
f(x) := x^2 + bx + c \in F[x].
\end{equation*}  
when $f(x)$ is irreducible.  
Recall from the introduction that $D_m(t, n)$ are the Dickson polynomials.

\begin{theorem} \label{power-criterion}
    Let $f(x) := x^2 + bx + c \in F[x]$ be irreducible with a root $\alpha \in F(\sqrt{\Delta}) \setminus F$, where $\Delta := b^2 - 4c$. Then $\alpha \in F(\sqrt{\Delta})^m$, for some positive integer $m$ if and only if there exists some $n, t \in F$ with $c = n^m$ and $b = -D_m(t, n)$.
\end{theorem}
 Before proving Theorem~\ref{power-criterion}, we need to establish two theorems about polynomials of the form $f(x^m)$.  Given variables $\beta, \beta_1$, define:
\begin{equation*}
    t := \beta + {\beta_1},\text{\qquad} n := \beta{\beta_1}
\end{equation*}
For $m \in \mathbb{N}$, define
\begin{equation*}
\Psi_{m, \beta}(x) := \frac{x^m - \beta^m}{x - \beta}
\end{equation*}
and define the polynomial $p_m(x) := \Psi_{m, \beta}(x)\Psi_{m, \beta_1}(x)$.  For $m>0$, let  $a_{m}$ be the coefficient of $x^{m-1}$ in  $p_m(x)$

\begin{theorem} \label{recursion}
    The $a_m$ are given by $a_1 = 1$, $a_2 = t$, and the recursive formula
    \begin{equation*}
        a_{m + 2} = ta_{m + 1} - na_m
    \end{equation*}
    and $p_m(x)$ is expressible in terms of $\{a_i\}_{1 \leq i \leq m}$ by: $$p_m(x) = \sum_{i = 1}^{m - 1} a_ix^{2m - 1 - i} + a_mx^{m - 1} + \sum_{i = 1}^{m - 1} a_{m - i}n^ix^{m - 1 - i}$$ (where the outer sums are defined to be 0 in the case of $m = 1$).
\end{theorem}

\begin{proof}
    We first compute $p_m(x), a_m$ from their definition.
    \begin{equation*}
    \begin{aligned}
        p_m(x) &= \Psi_{m, \beta}(x)\Psi_{m, \beta_1}(x) = (\sum_{i = 0}^{m - 1} \beta^{m - 1 - i}x^i)(\sum_{j = 0}^{m - 1}
        {\beta_1}^{m - 1 - j}x^j) \\
        &= \sum_{l = 0}^{2(m - 1)} \Big{(}\sum_{\substack{i + j = l \\ 0 \leq i, j \leq m - 1}}\beta^{m - 1 - i} {\beta_1}^{m - 1 - j} \Big{)} x^l \\
        \end{aligned}
    \end{equation*} The coefficient of the $x^{m-1}$ term is 
    \begin{equation*}
        \displaystyle{a_m = \sum_{\substack{i + j = m - 1 \\ 0 \leq i, j \leq m - 1}}\beta^{m - 1 - i} {\beta_1}^{m - 1 - j} = \sum_{i = 0}^{m - 1} \beta^{m - 1 - i} {\beta_1}^i}.
    \end{equation*}
    
\noindent
We can then verify the base cases explicitly:
    \begin{equation*}
    \begin{aligned}
        p_1(x) & = \Psi_{1, \beta}(x) {\Psi_{1, \beta_1}(x)} = 1\cdot 1 =1  = a_1 \\
        p_2(x) & = \Psi_{2, \beta}(x) {\Psi_{2,\beta_1}(x)} = (x+\beta)(x+\beta_1)\\ &= x^2 + tx + n  \\&
        = a_1x^2 + a_2x^1 + a_1n^1x^0 \\
    \end{aligned}
    \end{equation*}

    With regard to the recursive relation of the $(a_m)$, we may compute:
    \begin{equation*}
    \begin{split}
        ta_m - na_{m - 1} & = (\beta + {\beta_1})  \sum_{\substack{i + j = m - 1 \\ 0 \leq i, j \leq m - 1}}\beta^{m - 1 - i}{\beta_1}^{m - 1 - j} \\
        &\text{\hspace*{.5in}}-  \beta{\beta_1}\sum_{\substack{i + j = m - 2 \\ 0 \leq i, j \leq m - 2}}\beta^{m - 2 - i}{\beta_1}^{m - 2 - j} \\
        & = \sum_{i = 0}^{m - 1} \beta^{m - i} {\beta_1}^i + \sum_{i = 0}^{m - 1} \beta^{m - 1 - i} {\beta_1}^{i + 1} - \sum_{i = 0}^{m - 2} \beta^{m - 1 - i} {\beta_1}^{i + 1} \\
        & = \sum_{i = 0}^{m - 1} \beta^{m - i} {\beta_1}^i + \beta_1^0{\beta}^m = \sum_{i = 0}^{m} \beta^{m - i} = a_{m + 1} \qquad 
    \end{split}
    \end{equation*}

    And with regard to the expression of $p_m$, we assume for induction that the identity holds for $p_m$ and expand $p_{m + 1}(x)$:

    \begin{equation*}
    \begin{aligned}
        p_{m + 1}(x) & = \sum_{l = 0}^{2m} \left( \sum_{\substack{i + j = l \\ 0 \leq i, j \leq m}}\beta^{m - i}{\beta_1}^{m - j} \right) x^l \allowdisplaybreaks[3]\\
        & = \sum_{l = m + 1}^{2m} \Big{(} \sum_{\substack{i + j = l \\ 1 \leq i, j \leq m}}\beta^{m - i}{\beta_1}^{m - j} \Big{)} x^l + \sum_{\substack{i + j = m \\ 0 \leq i, j \leq m}}\beta^{m - i}{\beta_1}^{m - j} x^m  \allowdisplaybreaks[3]\\
        & \text{\hspace*{1in}}+ \sum_{l = 0}^{m - 1} \Big{(} \sum_{\substack{i + j = l \\ 0 \leq i, j \leq m - 1}}\beta^{m - i}{\beta_1}^{m - j} \Big{)} x^l \allowdisplaybreaks[3] 
    \end{aligned}
    \end{equation*}
        
    \begin{equation*}
    \begin{aligned}
        \text{\hspace*{.5in}}
        & = x^2 \sum_{l = m - 1}^{2m - 2} \Big{(} \sum_{\substack{i + j = l \\ 0 \leq i, j \leq m - 1}}\beta^{m - 1 - i}{\beta_1}^{m - 1 - j} \Big{)} x^l + a_{m + 1}x^m \\
        & \text{\hspace*{1in}} + \beta {\beta_1}\sum_{l = 0}^{m - 1} \Big{(} \sum_{\substack{i + j = l \\ 0 \leq i, j \leq m - 1}}\beta^{m - 1 -i}{\beta_1}^{m - 1 - j} \Big{)} x^l \\
    \end{aligned}
    \end{equation*}
    \begin{equation*}
    \begin{aligned}
    \text{\hspace*{.2in}}&
    = n \Big{(} \sum_{i = 1}^{m - 1} a_{m - i}n^ix^{m - 1 - i} + a_mx^{m - 1} \Big{)} + a_{m + 1}x^m \\
        &\text{\hspace*{1in}}+ x^2 \Big{(} a_mx^{m - 1} + \sum_{i = 1}^{m - 1} a_ix^{2m - 1 - i} \Big{)} \\
        \end{aligned}
    \end{equation*}
    \begin{equation*}
    \begin{aligned}
        \text{\hspace*{.5in}}
        & = \sum_{i = 1}^{m} a_{m + 1 - i}n^ix^{m - i} + a_{m + 1}x^m + \sum_{i = 1}^{m} a_ix^{2m + 1 - i} 
    \end{aligned} 
    \end{equation*} 
\end{proof}

The following theorem 
will be a tool to study polynomials of the form $f(x^m)$ and prove Theorem~\ref{power-criterion}.

\begin{theorem} \label{general-Dickson}
    Use the same notation as in Theorem \ref{recursion}, and define the polynomial $g(x) := (x - \beta^m)(x - \beta_1^m)$. Then $g(x^m) = x^{2m} - D_m(t, n)x^m + n^m$, where $D_m$ is the $m$th Dickson polynomial of the first kind.
\end{theorem}

\begin{proof}
    \begin{equation*}
    \begin{split}
        g(x^m) & = (x^m - \beta^m)(x - \beta_1^m) = (x - \beta)\Psi_{m, \beta}(x)(x - \beta_1)\Psi_{m, \beta_1}(x) \\& = (x^2 - tx + n)p_m(x) \\
        & = (x^2 - tx + n) \left( \sum_{i = 1}^{m - 1} a_ix^{2m - 1 - i} + a_mx^{m - 1} + \sum_{i = 1}^{m - 1} a_{m - i}n^ix^{m - 1 - i} \right) \\
        & = a_1x^{2m} + (na_{m - 1} - ta_m + na_{m - 1})x^m + a_1n^m \\
        & = x^{2m} + (2na_{m - 1} - ta_m)x^m + n^m \\
    \end{split}
    \end{equation*}

    \noindent
    We define 
    \begin{equation*}
       b_m = b_m(\beta, \beta_1)= 2na_{m - 1} - ta_m, 
    \end{equation*} 
    Then $g(x^m)= x^{2m} + b_mx^m + n^m$. We observe now that:
    \begin{equation*}
    \begin{split}
        tb_{m + 1} - nb_m & = t(2na_m - ta_{m + 1}) - n(2na_{m - 1} - ta_m) \\
        & = 2n(ta_m - na_{m - 1}) - t(ta_{m + 1} - na_m) \\&= 2na_{m + 1} - ta_{m + 2} \\&= b_{m + 2}
    \end{split}
    \end{equation*}

    Also, $b_1 = -t$ and $b_2 = 2n - t^2$. (Note: we define $a_0 := 0$ so that $b_1 = -t$, in agreement with the direct expansion of $g(x^m)$.) So indeed $b_m = -D_m(t, n)$, and as claimed, $g(x^m) = x^{2m} - D_m(t, n)x^m + n^m$.
\end{proof}

We now use Theorem \ref{general-Dickson} and Capelli's Theorem to study polynomials of the form $f(x^m)$ and prove Theorem~\ref{power-criterion}.

\begin{proof}[Proof of Theorem~\ref{power-criterion}]
    If there exists $\beta \in F(\sqrt{\Delta})$ with $\beta^m = \alpha$, $x^m  - \alpha = (x - \beta)\Psi_m(x)$ is a valid factorization of $x^m - \alpha \in F(\sqrt{\Delta})[x]$. We now identify the symbol $\beta_1$ from Theorems \ref{recursion} and \ref{general-Dickson} with the conjugate of $\beta$, namely $\mathrm{N}(\beta)/\beta \in F(\sqrt{\Delta})$. From Capelli's Theorem (and the multiplicativity of norms), this factorization induces the $F[x]$-factorization $$x^{2m} + bx^m + c = \mathrm{N}(x - \beta)\mathrm{N}(\Psi_{m, \beta}(x)) = (x^2 - tx + n)p_m(x)$$ Equating coefficients with the expression $f(x^m) = x^{2m} - D_m(t, n)x^m + n^m$ from Theorem \ref{general-Dickson} gives the desired equality:
    $$b = -D_m(t, n), c = n^m $$ with $$t = \mathrm{Tr}(\beta), n = \mathrm{N}(\beta) \in F$$
  
\noindent
Conversely, if $\exists n, t \in F$ with the given conditions, then by Theorem \ref{general-Dickson}, $$f(x^m) = (x^2 - tx + n)p_m(x)$$ By Capelli's Theorem, $(x^2 - tx + n)$ is the product of the norms of some irreducible polynomials in $F(\alpha) = F(\sqrt{\Delta})$ that divide $x^m - \alpha$. $\alpha$ is quadratic, so $x^2 - tx + n = \mathrm{N}(x - \beta)$ for some $(x - \beta) \mid (x^m - \alpha)$, and $\alpha = \beta^m \in F(\sqrt{\Delta})^m$.
\end{proof}

The criterion for reducibility, Theorem \ref{reducibility-criterion}, follows as a corollary.

\begin{proof}[Proof of Theorem \ref{reducibility-criterion}]
    We use the notation common to Theorems \ref{reducibility-criterion} and \ref{power-criterion}. If $f$ is reducible, so is $h = f(x^m)$. If $f$ is irreducible, from Capelli's Theorem, $h$ is reducible if and only if $x^m - \alpha$ is reducible in $F(\alpha)$, where $\alpha$ is a root of $f$. From Theorem~\ref{cyclotomic}, this occurs if and only if $\alpha \in F(\alpha)^p$ for some $p \mid m$ prime, or $\alpha \in -4F(\alpha)^4$ and $4 \mid m$. From Theorem \ref{power-criterion}, $\alpha \in F(\alpha)^p$ if and only if there exist $n, t \in F$ with $c = n^p$ and $b = -D_p(t, n)$. If $\alpha = -4\beta^4$ for some $\beta \in F(\alpha)$, denote the conjugate of $\beta$ over $F$ by $\beta_1$, with $\alpha_1 = -4\beta_1^4$. Then
    
    \begin{equation*}
    \begin{aligned}
        x^4 + bx^2 + c & = (x^2 - \alpha)(x^2 - \alpha_1) \\&= (x^2 + 4\beta^4)(x^2 + 4\beta_1^4) \\
        & = x^4 + 4(\beta^4 + \beta_1^4)x^2 + 16\beta^4\beta_1^4 \\
        &= x^4 + 4D_4(\mathrm{Tr}(\beta), \mathrm{N}(\beta)) + 16(\mathrm{N}(\beta))^4\\
    \end{aligned}
    \end{equation*}
    
\noindent
By Waring's identity (Remark \ref{Waring}). And indeed $t := \mathrm{Tr}(\beta) \in F$, $n := \mathrm{N}(\beta) \in F$. Conversely, if such $t, n$ exist, we observe that
    \begin{equation*}
    \begin{split}
        x^8 + bx^4 + c & = x^8 + 4D_4(t, n) + 16n^4 \\& = x^8 + (4t^4 - 16nt^2 + 8n^2)x^4 + 16n^4 \\
        & = (x^4 + (2t^2 - 4n)x^2 - 4n^2)(x^4 - (2t^2 - 4n)x^2 - 4n^2) \\
    \end{split}
    \end{equation*}

\noindent
So $f(x^4)$ and therefore $f(x^m)$ are reducible.
\end{proof}

\begin{remark}
    We note that the criterion of Theorem \ref{reducibility-criterion} is very similar to parts (vi) and (vii) of Theorem 6 of Schinzel \cite{schinzel-dissertation}. 
    It is quite possible that Theorem \ref{reducibility-criterion} is encompassed by the results in   \cite{schinzel-dissertation} on the reducibility of trinomials. However, those results utilize elliptic curves and the proof of Theorem \ref{reducibility-criterion} does not. 
\end{remark}

\section{Properties of $h = x^{2p} + bx^p + c^p \in \Q[x]$ and $d_h = D_p(x, c) + b$}\label{sect-h}

We now restrict ourselves to the case when $F = \mathbb{Q}$, $m= p > 3$ is prime, and the constant term of $f$ is the form $c^p$. With 
\begin{equation*}
\begin{aligned}
    f(x)&=x^2+bx+c^p\\
    h(x)&= f(x^p) = x^{2p} + bx^p + c^p
\end{aligned}
\end{equation*}
we assume in this section and the next 
that  $h(x)$ is irreducible. By Theorem \ref{reducibility-criterion}(1), $\Delta := b^2 - 4c^p$ is not a rational square and the splitting field of $f$ is $\mathbb{Q}(\sqrt{\Delta})$. Also,  by Theorem~\ref{reducibility-criterion}(2), the irreducibility of $h$ implies that $b\neq -D_p(t, c)$ for any $t\in \Q$.  Hence the polynomial $d_h:= D_p(x, c) + b$ has no rational roots.

Let $\zeta_n$ be a primitive $n$th root of unity and let 
\begin{equation*}
F_n = \Q(\zeta_n)
\end{equation*} be the $n$th cyclotomic field. We will denote the splitting field of $h$ by $K$ and the splitting field of $d_h$ by $L$. We denote the roots of $f$ by $\alpha, \alpha_1$. We then choose $\beta$ so that $\beta^p = \alpha$ and $\beta_1 := \frac{c}{\beta}$ (and indeed $\beta_1^p = \frac{c^p}{\beta^p} = \frac{c^p}{\alpha} = \alpha_1$). $h$ has $2p$ roots, which are precisely $$\{\zeta_p^i\beta \}_{i \in \mathbb{Z}_p} \cup \{\zeta_p^i \beta_1\}_{i \in \mathbb{Z}_p}$$ We claim now:

\begin{lemma} \label{B_i-distinct}
    $\mathbb{Q}(\zeta_p^i\beta) = \mathbb{Q}(\zeta_p^{-i}\beta_1)$, and moreover, these are the only roots of $h(x)$ in this extension of $\mathbb{Q}$.
\end{lemma}

\begin{proof}
    By construction $\zeta_p^{-i}\beta_1 = \zeta_p^{-i} \frac{c}{\beta} = \frac{c}{\zeta_p^i\beta}$, so $\Q(\zeta_p^i \beta) = \mathbb{Q}(\zeta_p^{-i}\beta_1)$. Suppose that some other root of $h$ lies in $\mathbb{Q}(\zeta_p^i\beta)$, without loss of generality, $\zeta_p^j\beta$. Then $\zeta_p^{i - j} \in \mathbb{Q}(\zeta_p^i\beta)$. If $i \neq j$, $F_p \subset \mathbb{Q}(\beta)$. This is impossible because $[F_p : \mathbb{Q}] = p - 1$ and $[\mathbb{Q}(\beta): \mathbb{Q}] = 2p$, but $p - 1 \nmid 2p$ given that $p > 3$. So $i = j$, and $\zeta_p^j \beta = \zeta_p^i\beta$, which is one of the given roots after all.
\end{proof}

As a corollary, there are no double roots: $\alpha \neq \alpha_1$ else $f$ is reducible, and the only pairs of roots which generate the same extension are of the form $\zeta_p^i\beta, \zeta_p^{-i}\beta_1$, whose $p$th powers are $\alpha, \alpha_1$, respectively.

Now denote $B_i := \mathbb{Q}(\zeta_p^i\beta) = \mathbb{Q}(\zeta_p^{-i}\beta_1)$ and $\mathcal{B} := \{B_i\}_{i \in \mathbb{Z}_p}$, with $|\mathcal{B}| = p$. Clearly, $[B_i : \mathbb{Q}] = 2p$. Also, $\zeta_p^i\beta + \zeta_p^{-i}\beta_1 \in B_i$.  We show that it is a root of $d_h$.

\begin{lemma}
    $d_h(\zeta_p^i\beta + \zeta_p^{-i}\beta_1) = 0$.
\end{lemma}

\begin{proof}
    By factoring $f$, we see that: $$x^{2p} + bx^p + c^p = f(x^p) = (x^p - (\zeta_p^i\beta)^p)(x^p - (\zeta_p^{-i}\beta_1)^p)$$ And by Theorem \ref{general-Dickson}: $$f(x^p) = x^{2p} - D_p(\zeta_p^i\beta + \zeta_p^{-i}\beta_1, c)x^p + c^p$$ Equating coefficients, $-D_p(\zeta_p^i\beta + \zeta_p^{-i}\beta_1, c) = b$, i.e., $d_h(\zeta_p^i\beta + \zeta_p^{-i}\beta_1) = 0$.
\end{proof}

We now define the fields $D_i := \mathbb{Q}(\zeta_p^i\beta + \zeta_p^{-i}\beta_1)$, with $D_i \subset B_i$, and define \begin{equation*}
\mathcal{D} := \{D_i\}_{i \in \mathbb{Z}_p}
\end{equation*}
\textit{A priori}, we do not know that the fields $D_i$ are distinct, but for now, it suffices that the distinct symbols $D_i$ are in bijective correspondence with $\mathbb{Z}_p$ and $\mathcal{B}$. Before examining the fields $D_i$, it will be useful to prove that:

\begin{lemma} \label{transitivity}
    $\mathrm{Gal}(K / \mathbb{Q})$ acts transitively and equivalently on $\mathcal{B}$, on $\mathcal{D}$, and on the roots of $d_h$.
\end{lemma}

\begin{proof}
    Let $\sigma \in \mathrm{Gal}(K / \mathbb{Q})$. $\sigma$ permutes the roots of $h$, and from Lemma \ref{B_i-distinct}, it must be that $\sigma(\zeta_p^i\beta) \in B_j$ for some $j$. Then $\sigma(B_i) = B_j$, and $\sigma$ acts on $\mathcal{B}$. $\mathrm{Gal}(K / \mathbb{Q})$ is transitive on the roots of irreducible $h$, so its action on $\mathcal{B}$ is also transitive. $\sigma$ maps the pair of roots in $B_i$ to the pair of roots in $B_j$, so the action on $\mathcal{B}$ is equivalent to the action on the set of pairs of roots of $h$. Then: $$\sigma(\zeta_p^i\beta + \zeta_p^{-i}\beta_1) = \zeta_p^j\beta + \zeta_p^{-j}\beta_1$$ So $\sigma(D_i) = D_j$. Therefore, $\sigma$ acts equivalently on the roots of $d$, on $\mathcal{D}$, and on $\mathcal{B}$, according to the correspondence $\zeta_p^i\beta + \zeta_p^{-i}\beta_1 \in D_i \subset B_i$, and these equivalent actions of $\mathrm{Gal}(K / \mathbb{Q})$ are transitive.
\end{proof}

We can now prove the following properties about the fields $D_i$.

\begin{lemma} \label{D_i-properties}
    $d_h(x) = D_p(x, c) + b$ is irreducible, so $D_i \simeq \mathbb{Q}[x]/(d_h(x))$. Also, $B_i = D_i(\sqrt{\Delta})$, and $D_i = D_j$ if and only if $i = j$.
\end{lemma}

\begin{proof}
    $$x^{2p} + bx^p + c^p = f(x^p) = (x^p - \alpha)(x^p - \alpha_1) = (x^p - (\zeta_p^i\beta)^p)(x^p - (\zeta_p^{-i}\beta_1)^p)$$ By Theorem \ref{general-Dickson}, $b = -D_p(\zeta_p^i\beta + \zeta_p^{-i}\beta_1, c)$, so $\zeta_p^i\beta + \zeta_p^{-i}\beta_1 \in B_i$ is a root of $d_h$. Thus, $[D_i : \mathbb{Q}] \leq p$, and since $D_i \subset B_i$, $[D_i : \mathbb{Q}] \mid 2p$. Thus, $[D_i : \mathbb{Q}]$ could be $1$, $2$, or $p$. $[D_i : \mathbb{Q}] \neq 1$ because this would give a rational root for $d_h$, a contradiction.
    
    Now assume temporarily that $[D_i : \mathbb{Q}] = 2$. Then $D_i(\sqrt{\Delta}) \subset B_i$ is an extension of either degree 2 or degree 4. $4 \nmid 2p$ given $p > 3$, so $[D_i(\sqrt{\Delta}) : \mathbb{Q}] = 2$, and $D_i = \mathbb{Q}(\sqrt{\Delta})$. By the transitive action of the Galois group, all $D_i$ are isomorphic and are thus quadratic extensions. So each root of $d_h$ is a root of a quadratic factor, and $d_h$ factors as a product of quadratics. But this would make $\deg(d_h) = p$ even, a contradiction. So $[D_i : \mathbb{Q}] = p$ after all, and consequently, $d_h(x)$ is irreducible with $D_i \simeq \mathbb{Q}[x]/(d_h(x))$, as claimed. 
    
    Then $\sqrt{\Delta} \notin D_i$ because $2 \nmid p$, so $[D_i(\sqrt{\Delta}) : \mathbb{Q}] = 2p = [B_i : \mathbb{Q}]$ and $D_i(\sqrt{\Delta}) \subset B_i$, so $D_i(\sqrt{\Delta}) = B_i$. If $i \neq j$ but $D_i = D_j$, then $B_i = B_j$, a contradiction of Lemma \ref{B_i-distinct}. So $D_i = D_j$ if and only if $i = j$, as desired.
\end{proof}

As a technical lemma, we must now note that: 

\begin{lemma} \label{splitting-degree}
    $K = F_p(\beta)$. $[K : \Q] = 2p(p - 1)$ if $\sqrt{\Delta} \notin F_p$ and $[K : \Q] = p(p - 1)$ if $\sqrt{\Delta} \in F_p$.
\end{lemma}

\begin{proof}
    $F_p(\beta) = B_0(\zeta_p)$ certainly contains all the roots of $h$. Conversely, the roots of $h$ include $\beta$ and $\zeta_p\beta$, so the splitting field is at least $\Q(\zeta_p, \beta) = F_p(\beta)$. Thus, $K = F_p(\beta)$. 

    $[B_i : \Q] = 2p$ and $[F_p : \Q] = p - 1$, and $\gcd(2p, p - 1) = 2$ because $p$ is odd. Thus, $p(p - 1) \mid [K : \Q] \mid 2p(p - 1)$. If $\sqrt{\Delta} \in F_p$, then $[F_p : \Q(\sqrt{\Delta})] = \frac{p - 1}{2}$ is coprime to $[B : \Q(\sqrt{\Delta})] = p$. Thus $[K : \Q] = 2[K : \Q(\sqrt{\Delta})] = 2 \cdot \frac{p - 1}{2} \cdot p = p(p - 1)$. If $\sqrt{\Delta} \notin F_p$, then $[F_p : \Q(\sqrt{\Delta})] = p - 1$ is again coprime to $[B : \Q(\sqrt{\Delta})] = p$, and similarly $[K : \Q] = 2[K : \Q(\sqrt{\Delta})] = 2 \cdot (p - 1) \cdot p = 2p(p - 1)$.
\end{proof}

\begin{remark} \label{characteristic}
    As $\mathbb{Q}(\sqrt{\Delta})/\Q$ and $F_p/\Q$ are Galois extensions, an automorphism $\sigma \in \mathrm{Gal}(K / \mathbb{Q})$ acts as an automorphism of $\mathbb{Q}(\sqrt{\Delta}), F_p$. 
\end{remark}

Using the lemmas of this section, we are now able to specify the actions of  $\sigma\in \mathrm{Gal}(K/\Q)$ in terms of its action on $D_0, \sqrt{\Delta}$, and $\zeta_p$.  We define $\eps_\sigma=
\sigma(\sqrt{\Delta})/\sqrt{\Delta}=\pm 1$. 

\begin{theorem} \label{galois-actions}
    If $\sqrt{\Delta} \notin F_p$, then there exists a bijection between $\mathrm{Gal}(K / \mathbb{Q})$ 
    and $\mathcal{D} \times \{\pm 1\} \times \{\zeta_p^i\}_{i \in \mathbb{Z}_p^\times}$ given by  
    $\sigma \mapsto (\sigma(D_0), \eps_{\sigma}, \sigma(\zeta_p))$. 
    If $\sqrt{\Delta} \in F_p$,
    then there exists a bijection between $\mathrm{Gal}(K / \mathbb{Q})$ and $\mathcal{D} \times \{\zeta_p^i\}_{i \in \mathbb{Z}_p^\times}$ given by $\sigma \mapsto(\sigma(D_0), \sigma(\zeta_p))$.
\end{theorem}

\begin{proof}
    $\sigma \in \mathrm{Gal}(K / \mathbb{Q})$ permutes the roots of $h$. Because all roots of $h$ are expressible in terms of $\beta$ and $\zeta_p$ (as $\beta_1 = c/\beta$), $\sigma$ is determined entirely by $\sigma(\beta)$ and $\sigma(\zeta_p)$.
    
    If $\sqrt{\Delta} \notin F_p$, there are $2p$ choices for $\sigma(\beta)$ and $p - 1$ choices for $\sigma(\zeta_p)$, so $|\mathrm{Gal}(K / \mathbb{Q})| \leq 2p(p - 1)$. But $|\mathrm{Gal}(K / \mathbb{Q})| = 2p(p - 1)$, so all choices must correspond to distinct elements of the Galois group. The action $\sigma(\beta)$ determines $(\sigma(B_0), \eps_\sigma)$, and vice-versa, and the actions on $B_0, D_0$ are also equivalent. Therefore, $\sigma \in \mathrm{Gal}(K / \mathbb{Q})$ corresponds exactly to a choice of 
    $(D_i, \pm 1, \zeta^j_p)$. 
    \qed

    If $\sqrt{\Delta} \in F_p$, $\sigma(\zeta_p)$ determines $\sigma(\sqrt{\Delta})$. Therefore, $(\sigma(\zeta_p), \sigma(D_0))$ determines $\sigma(\beta)$ and thus $\sigma$. There are $p$ choices for $\sigma(D_0)$ and $p - 1$ choices for $\sigma(\zeta_p)$, so $|\mathrm{Gal}(K / \mathbb{Q})| \leq p(p - 1)$. But $|\mathrm{Gal}(K / \mathbb{Q})| = p(p - 1)$, so all choices must correspond to distinct elements of the Galois group. Therefore, $\sigma \in \mathrm{Gal}(K / \mathbb{Q})$ corresponds exactly to a choice of $(D_i, \zeta_p^j)$.
\end{proof}

\section{The Galois groups of $h$ and $d_h$}

As in Section~\ref{sect-h}, we assume that
\begin{equation*}
    h(x)= f(x^p) = x^{2p} + bx^p + c^p,
\end{equation*}
that  $h(x)$ is irreducible and
let $d_h$ denote
the polynomial $d_h= D_p(x, c) + b$.  The assumption on $h$ shows that $d_h$ is irreducible. 

\begin{theorem} \label{d-split}
Let $L$ be the splitting field of $d_h$ over $K$.  Then $K=L(\sqrt{\Delta})$ and
for all $i$, 
    \begin{equation*}
    D_i(\zeta_p + \zeta_p^{-1}, \sqrt{\Delta}(\zeta_p - \zeta_p^{-1}))\subset L 
    \end{equation*}
\end{theorem}

\begin{proof}
    $L$ contains all of the roots of $d_h$, 
    which splits in $K$, so $D_i \subset L \subset K$ for each $i$. 
    Since $\zeta_p^i \beta +
    \zeta_p^{-i} \beta_1\in D_i$, each term is in $L$.
    Then $L$ contains $$\frac{(\zeta_p\beta + \zeta_p^{-1}\beta_1) + (\zeta_p^{-1}\beta + \zeta_p\beta_1)}{(\zeta_p^0\beta + \zeta_p^{-0}\beta_1)} = \frac{(\beta + \beta_1)(\zeta_p + \zeta_p^{-1})}{\beta + \beta_1} = \zeta_p + \zeta_p^{-1}$$ as well.
    
    Now choose an arbitrary $\sigma \in \mathrm{Gal}(K / \mathbb{Q})$ which fixes $L$. $\sigma$ acts trivially on $\mathcal{D}$ and $\mathcal{B}$ and acts on $\sqrt{\Delta}$ either trivially or by conjugation. If $\sigma(\sqrt{\Delta}) = \sqrt{\Delta}$, then $\sigma(\zeta_p^i\beta) = \zeta_p^i\beta$, so $\sigma(\zeta_p) = \zeta_p$. Thus:
    
    \begin{equation*}
    \begin{aligned}
        \sigma(\sqrt{\Delta}(\zeta_p - \zeta_p^{-1})) &= \sigma(\sqrt{\Delta})(\sigma(\zeta_p) - (\sigma(\zeta_p))^{-1})\\
        &= \sqrt{\Delta}(\zeta_p - \zeta_p^{-1})
        \end{aligned}
    \end{equation*}
    
    Similarly, if $\sigma(\sqrt{\Delta}) = -\sqrt{\Delta}$, then $\sigma(\zeta_p^i\beta) = \zeta_p^{-i}\beta_1$, so $\sigma(\zeta_p) = \zeta_p^{-1}$. Thus,

    \begin{equation*}
    \begin{aligned}
        \sigma(\sqrt{\Delta}(\zeta_p - \zeta_p^{-1})) &= \sigma(\sqrt{\Delta})(\sigma(\zeta_p) - (\sigma(\zeta_p))^{-1})\\
        &= -\sqrt{\Delta}(\zeta_p^{-1} - \zeta_p)\\ 
        &= \sqrt{\Delta}(\zeta_p - \zeta_p^{-1})
        \end{aligned}
    \end{equation*}

    Because $\sqrt{\Delta}(\zeta_p - \zeta_p^{-1})$ is fixed by the subgroup fixing $L$, $\sqrt{\Delta}(\zeta_p - \zeta_p^{-1}) \in L$. Finally, $L(\sqrt{\Delta}) \subset K$ contains each $B_i = D_i(\sqrt{\Delta})$, so $K=L(\sqrt{\Delta})$.
\end{proof}

From this point, we will need to construct automorphisms using Theorem \ref{galois-actions}, rather than considering arbitrary automorphisms as we have done in the previous section. We can prove:

\begin{theorem} \label{dickson-split}
    The splitting field $L$ of $d_h$ is $D_0(\zeta_p + \zeta_p^{-1}, \sqrt{\Delta}(\zeta_p - \zeta_p^{-1}))$.
\end{theorem}

To help to understand the proof of Theorem~\ref{dickson-split}, Figures~\ref{fig:1},\,~\ref{fig:2},~\ref{fig:3} provide diagrams of the relevant field inclusions in each case. All of the boxed fields are Galois over $\Q$, and the fields marked $``p \times"$ are conjugates.

\begin{proof}
    $D_0(\zeta_p + \zeta_p^{-1}, \sqrt{\Delta}(\zeta_p - \zeta_p^{-1})) \subset L$ from Theorem \ref{d-split}. Since $$K = B_0(\zeta_p) = D_0(\zeta_p, \sqrt{\Delta}) \subset D_0(\zeta_p + \zeta_p^{-1}, \sqrt{\Delta}(\zeta_p - \zeta_p^{-1}), \sqrt{\Delta})$$ and $\sqrt{\Delta}$ is of degree at most 2, we see that that $$[K : L] \leq [K : D_0(\zeta_p + \zeta_p^{-1}, \sqrt{\Delta}(\zeta_p - \zeta_p^{-1}))] \leq 2$$
    \noindent
    Recall that $F_p=\Q(\zeta_p)$ is the cyclotomic field. There are then 3 cases: 
    \begin{enumerate}
    \item\label{case-1} $\sqrt{\Delta} \notin F_p$; or 
    \item\label{case-2} $\sqrt{\Delta} \in F_p$ and  $p \equiv 1 \bmod 4$; or
    \item\label{case-3} $\sqrt{\Delta} \in F_p$ and $p \equiv 3 \bmod 4$. 
    \end{enumerate}    
    \emph{Case \eqref{case-1}}: Consider the case of $\sqrt{\Delta} \notin F_p$.
    From Theorem \ref{galois-actions}, there exists an element of $\sigma\in \mathrm{Gal}(K / \mathbb{Q})$ defined by $(\sigma(D_0), \epsilon_\sigma, \sigma(\zeta_p)) = (D_0, -1, \zeta_p^{-1})$. Then $$\sigma(\zeta_p^i\beta + \zeta_p^{-i}\beta_1) = \sigma(\zeta_p)^i\sigma(\beta) + \sigma(\zeta_p)^{-i}\sigma(\beta_1) = \zeta_p^{-i}\beta_1 + \zeta_p^i\beta$$ for all $i$, so $\sigma$ fixes all roots of $d_h$ and therefore $L$. The trivial element also fixes $L$, so by the Galois correspondence, $[K : L] \geq 2$. So $[K : L] = 2$, and $L = D_0(\zeta_p + \zeta_p^{-1}, \sqrt{\Delta}(\zeta_p - \zeta_p^{-1}))$ after all, and is of degree $\frac{1}{2}\cdot 2p(p - 1) = p(p - 1)$, using Lemma \ref{splitting-degree}.

    \emph{Case \eqref{case-2}}: Now suppose that $\sqrt{\Delta} \in F_p$ and $p \equiv 1 \bmod 4$. Then $[\Q(\zeta_p + \zeta_p^{-1}) : \Q] = \frac{p - 1}{2}$ is even, so $\Q(\sqrt{\Delta}) \subset \Q(\zeta_p + \zeta_p^{-1})$ by the Galois correspondence. Thus
    $$D_0(\zeta_p + \zeta_p^{-1}, \sqrt{\Delta}(\zeta_p - \zeta_p^{-1})) = D_0(\zeta_p + \zeta_p^{-1}, (\zeta_p - \zeta_p^{-1}), \sqrt{\Delta}) \supset B_0(\zeta_p) = K \supset L$$ and in fact $K = L = D_0(\zeta_p + \zeta_p^{-1}, \sqrt{\Delta}(\zeta_p - \zeta_p^{-1}))$, which by Lemma \ref{splitting-degree} is of degree $p(p - 1)$.
    
    \emph{Case \eqref{case-3}}: Finally, we suppose that $\sqrt{\Delta} \in F_p$, but $p \equiv 3 \bmod 4$. Then $[\Q(\zeta_p + \zeta_p^{-1}) : \Q] = \frac{p - 1}{2}$ is odd, so $\sqrt{\Delta} \notin \Q(\zeta_p + \zeta_p^{-1})$ by the Galois correspondence (else $\Q(\sqrt{\Delta}) \subset \Q(\zeta_p + \zeta_p^{-1})$). From Theorem \ref{galois-actions}, there exists an element $\sigma\in \mathrm{Gal}(K / \mathbb{Q})$ defined by $(\sigma(D_0), \sigma(\zeta_p)) = (D_0, \zeta_p^{-1})$, and consequently $\sigma(\sqrt{\Delta}) = -\sqrt{\Delta}$ because $\sqrt{\Delta} \notin \mathbb{Q}(\zeta_p + \zeta_p^{-1})$, the fixed field of conjugation in $F_p$. Then as in case (1), $\sigma$ fixes all roots of $d_h$ and therefore $L$, so $[K : L] \geq 2$. Then $[K : L] = 2$, and $L = D_0(\zeta_p + \zeta_p^{-1}, \sqrt{\Delta}(\zeta_p - \zeta_p^{-1})) = D_0(\zeta_p + \zeta_p^{-1})$, after all, and is of degree $\frac{p(p - 1)}{2}$, using Lemma \ref{splitting-degree}.
\end{proof}

\begin{remark}
    The choice of the automorphism $\sigma$ fixing $L$ follows Jones' construction in \cite[Section 3, p. 6]{reciprocal}.
\end{remark}

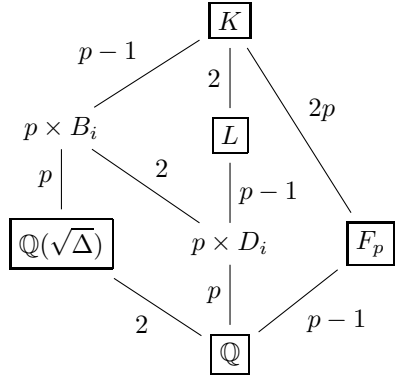
\begin{figure}
\begin{tikzpicture}[node distance=8mm and 8mm]
    \node (Q) {$\boxed{\mathbb{Q}}$} ;
    \node (D) [above = of Q] {$p \times D_i$} ;
    \node (delta) [left = of D] {$\boxed{\mathbb{Q}(\sqrt{\Delta})}$};
    
    \node (cyclotomic) [right= of D] {$\boxed{F_p}$} ;

    \node (L) [above = of D] {$\boxed{L}$} ;
    \node (B) [above = of delta] {$p \times B_i$} ;
    
    \node (K) [above = of L] {$\boxed{K}$} ;

    \draw (Q) edge node[left] {$p$} (D) ;
    \draw (Q) edge node[below left] {$2$} (delta) ;
    \draw (D) edge node[above right] {$2$} (B) ;
    \draw (delta) edge node[left] {$p$} (B) ;

    \draw (Q) edge node[below right] {$p - 1$} (cyclotomic) ;

    \draw (D) edge node[right] {$p - 1$} (L) ;

    \draw (B) edge node[above left] {$p - 1$} (K) ;

    \draw (L) edge node[left] {$2$} (K) ;
    \draw (cyclotomic) edge node[above right] {$2p$} (K) ;

\end{tikzpicture}
\caption{The field diagram when $\sqrt{\Delta} \notin F_p$}
\label{fig:1}
\end{figure}
\begin{figure}
\begin{tikzpicture}[node distance=4.5mm and 4.5mm]
    \node (Q) {$\boxed{\mathbb{Q}}$} ;
    \node (D) [above left = of Q] {$p \times D_i$} ;
    \node (delta) [above right = of Q] {$\boxed{\mathbb{Q}(\sqrt{\Delta})}$};
    \node (B) [above = of D] {$p \times B_i$} ;
    \node (half) [above = of delta] {$\boxed{\mathbb{Q}(\zeta_p + \zeta_p^{-1})}$} ;
    \node (double) [above = of B] {$p \times D_i(\zeta_p + \zeta_p^{-1})$} ;
    \node (cyclotomic) [above = of half] {$\boxed{F_p}$} ;
    \node (K) [above left = of cyclotomic] {$\boxed{K}$} ;

    \draw (Q) edge node[below left] {$p$} (D) ;
    \draw (Q) edge node[below right] {$2$} (delta) ;
    \draw (D) edge node[left] {$2$} (B) ;
    \draw (delta) edge node[above] {$p$} (B) ;
    \draw (delta) edge node[right] {$\frac{p - 1}{4}$} (half) ;
    \draw (B) edge node[left] {$\frac{p - 1}{4}$} (double) ;
    \draw (half) edge node[above] {$p$} (double) ;
    \draw (half) edge node[right] {$2$} (cyclotomic) ;
    \draw (double) edge node[above left] {$2$} (K) ;
    \draw (cyclotomic) edge node[above right] {$p$} (K) ;
\end{tikzpicture}
\caption{The field diagram when $\sqrt{\Delta} \in F_p$ and $p \equiv 1 \bmod 4$}
\label{fig:2}
\end{figure}
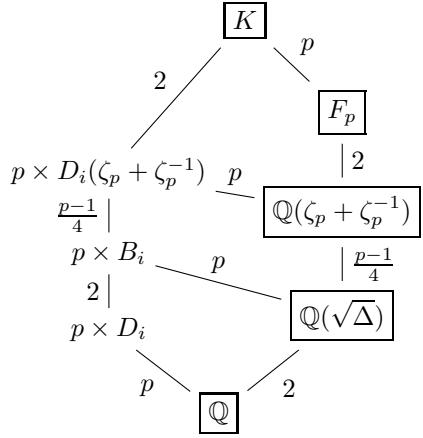

\begin{figure}
\begin{tikzpicture}[node distance=9mm and 9mm]
    \node (Q) {$\boxed{\mathbb{Q}}$} ;
    \node (delta) [above = of Q] {$\boxed{\mathbb{Q}(\sqrt{\Delta})}$};
    \node (D) [left = of delta] {$p \times D_i$} ;
    \node (B) [above = of D] {$p \times B_i$} ;
    \node (half) [right = of delta] {$\boxed{\mathbb{Q}(\zeta_p + \zeta_p^{-1})}$} ;

    \node (bicycle) [above = of half] {$\boxed{F_p}$} ;
    \node (candidate) [above = of delta] {$\boxed{L}$} ;
    \node (double) [above = of candidate] {$\boxed{K}$} ;

    \draw (Q) edge node[below left] {$p$} (D) ;
    \draw (Q) edge node[left] {$2$} (delta) ;
    \draw (D) edge node[left] {$2$} (B) ;
    \draw (delta) edge node[left] {$p$} (B) ;
    \draw (Q) edge node[below right] {$\frac{p - 1}{2}$} (half) ;
    \draw (B) edge node[left] {$\frac{p - 1}{2}$} (double) ;
    
    \draw (half) edge node[left] {$2$} (bicycle) ;
    \draw (delta) edge node[above] {$\frac{p-1}{2}$} (bicycle) ;
    \draw (bicycle) edge node[above] {$p$} (double) ;
    \draw (D) edge node[above] {$\frac{p-1}{2}$} (candidate) ;
    \draw (half) edge node[left] {$p$} (candidate) ;
    \draw (candidate) edge node[left] {$2$} (double);
\end{tikzpicture}
\caption{The field diagram when $\sqrt{\Delta} \in F_p$ and $p \equiv 3 \bmod 4$.}
\label{fig:3}
\end{figure}
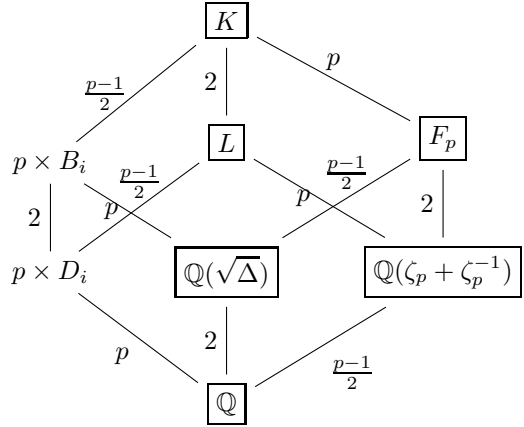

We are almost ready to prove Theorem \ref{galois-groups}, but first, we must recall the following fact about cyclotomic fields:

\begin{lemma}\cite[VI, Thm 3.3]{lang} \label{quadratic-cyclotomic}
    Let $p$ be an odd prime and let $(\frac{-1}{p}) = (-1)^{p(p - 1)/2}$ be the quadratic Legendre symbol. 
    Let $K=\mathbb{Q} \left( \sqrt{(\frac{-1}{p})p}\, \right)$. Then $K\subset \mathbb{Q}(\zeta_p)$ and $K$ is the only quadratic extension of $\Q$ that is a subfield of $\Q(\zeta_p)$.
\end{lemma}

\begin{corollary} \label{quadratic-cyclotomic-criterion}
    $\sqrt{\Delta} \in \mathbb{Q}(\zeta_p)$ if and only if $\Delta \in (-1)^{p(p - 1)/2}p\mathbb{Q}^2$.
\end{corollary}

Finally, we can prove Theorem \ref{galois-groups}.

\begin{proof}[Proof of Theorem \ref{galois-groups}]
    Recall that we are studying irreducible polynomials $h$ of the form $x^{2p} + bx^p + c^p$. We note that for any irreducible $h$ of this form, $\Delta := b^2 - 4c$ must fall into exactly one of the 3 cases given. And from Lemmas \ref{D_i-properties} and \ref{dickson-split}, $d_h$ is also irreducible with splitting field $L = \mathbb{Q}(\beta + \beta_1, \zeta_p + \zeta_p^{-1}, \sqrt{\Delta}(\zeta_p - \zeta_p^{-1}))$. We will use throughout the notation of $D_i := \mathbb{Q}(\zeta_p^i\beta + \zeta_p^{-i}\beta_1)$. Note throughout that while we will be examining specific elements of $\mathrm{Gal}(K / \mathbb{Q})$, we will be considering their actions on the roots of $d$, which is equivalent to the canonical actions of their images in $\mathrm{Gal}(L / \mathbb{Q})$ under the division map.\\
    
    \noindent
    \textit{Proof of Case 1.} If $b^2 - 4c \notin (-1)^{p(p-1)/2}p\mathbb{Q}^2$, then $\sqrt{\Delta} \notin F_p$ by Corollary \ref{quadratic-cyclotomic-criterion}. From above, this means that $L = D_0(\zeta_p + \zeta_p^{-1}, \sqrt{\Delta}(\zeta_p - \zeta_p^{-1}))$ is an extension of degree $p(p - 1)$ which does not contain $\sqrt{\Delta}$.
    
    Consider now the elements $\sigma, \tau \in \mathrm{Gal}(K / \mathbb{Q})$ defined using Theorem \ref{galois-actions} by:
    \begin{equation*}
    \begin{split}
        (\sigma(D_0), \epsilon_\sigma, \sigma(\zeta_p)) & = (D_0, 1, \zeta_p^r) \\
        (\tau(D_0), \epsilon_\tau, \tau(\zeta_p)) & = (D_1, 1, \zeta_p) \\
    \end{split}
    \end{equation*}
        
    Where $\zeta_p^r$ is a generator of the group of $p$th roots of unity. By inspection, $\sigma(D_i) = D_{ri}$ and $\tau(D_i) = D_{i + 1}$. Let $S_p$ be the symmetric group on the set $\Z_p=\{ 0, \dots, p-1\}$. The Galois group of $d_h$ therefore contains the elements
    \begin{equation*}
    \begin{split}
    \tau_1=&\ (1 \ r \ r^2 \ \cdots \ r^{p - 2}) \in S_p, \\
    \tau_2=&\ (0 \ 1 \ 2 \ \cdots \ p - 1) \in S_p
    \end{split}
    \end{equation*}
    The elements $\tau_1$, $\tau_2$  generate $\mathrm{Aff}(\mathbb{F}_p)$, whose size is $p(p - 1)$. As $|\mathrm{Gal}(L / \mathbb{Q})| = p(p - 1)$, $\mathrm{Gal}(L / \mathbb{Q}) \simeq \mathrm{Aff}(\mathbb{F}_p)$. And $\mathrm{Gal}(\mathbb{Q}(\sqrt{\Delta}) / \mathbb{Q}) \simeq C_2$, so the Galois group of $h$ is $$\mathrm{Gal}(K / \mathbb{Q}) \simeq \mathrm{Gal}(L / \mathbb{Q}) \times \mathrm{Gal}(\Q(\sqrt{\Delta} / \Q) \simeq \mathrm{Aff}(\mathbb{F}_p) \times C_2$$ \\

    \noindent
    \textit{Proof of Case 2.} If $p \equiv 1 \bmod 4$ and $b^2 - 4c \in p\mathbb{Q}^2$, then $\sqrt{\Delta} \in F_p$ by Corollary \ref{quadratic-cyclotomic-criterion}. From above, $L = K$ is the splitting field of $d_h$ and $h$, and $|\mathrm{Gal}(K / \mathbb{Q})| = |\mathrm{Gal}(L / \mathbb{Q})| = p(p - 1)$. Now consider elements $\sigma, \tau \in \mathrm{Gal}(K/\Q)$ defined using Theorem \ref{galois-actions} by:
    \begin{equation*}
    \begin{split}
        (\sigma(D_0), \sigma(\zeta_p)) & = (D_0, \zeta_p^r) \\
        (\tau(D_0), \tau(\zeta_p)) & = (D_1, \zeta_p) \\
    \end{split}
    \end{equation*}

    For $\zeta_p^r$ a generator. $\sigma$ conjugates $\sqrt{\Delta} \in F_p \setminus \Q$, so $\sigma(\zeta_p^i\beta) = \zeta_p^{ri} \beta_1$, and $\sigma(D_i) = D_{-ri}$. $\zeta_p^{-r}$ is a generator because 
    \begin{equation*}
        \left(\frac{-r}{p}\right) = 
        \left(\frac{-1}{p}\right)
        \left(\frac{r}{p}\right) = (1)(-1) = -1
    \end{equation*}
    
    So $\sigma$ will be a  $p - 1$-cycle. And again, $\tau(D_i) = D_{i + 1}$. The Galois group of $d_h$ therefore contains the elements

    \begin{center}
\begin{tabular}{c c c c c c c c c c}
$\tau_1$
&$=$
& $(1$ & 
$-r$ & $(-r)^2$&$\ldots$ & 
$(-r)^{(p-2)}$ & $)$ &
\\
$\tau_2$
&$=$&
$(0$ & 
$1$ & $2$&$\ldots$ & 
$p-1$ & $)$
\\

\end{tabular}

\end{center}
    
As in Case 1, the elements $\tau_1$, $\tau_2$ generate $\mathrm{Aff}(\mathbb{F}_p)$, and $|\mathrm{Aff}(\mathbb{F}_p)| = p(p - 1) = |\mathrm{Gal}(K / \mathbb{Q})| = |\mathrm{Gal}(L / \mathbb{Q})|$, so the Galois groups of $d_h, h$ are identically $$\mathrm{Gal}(K / \mathbb{Q}) \simeq \mathrm{Gal}(L / \mathbb{Q}) \simeq \mathrm{Aff}(\mathbb{F}_p)$$ \\

\noindent
\textit{Proof of Case 3.} If $p \equiv 3 \bmod 4$ and $b^2 - 4c \in -p\mathbb{Q}^2$, then $\sqrt{\Delta} \in F_p$ by Corollary \ref{quadratic-cyclotomic-criterion}. From above, $|\mathrm{Gal}(K / \mathbb{Q})| = p(p - 1)$, but $|\mathrm{Gal}(L / \mathbb{Q})| = \frac{p(p - 1)}{2}$. Again we consider elements $\sigma, \tau \in \mathrm{Gal}(K/\Q)$ defined using Theorem \ref{galois-actions} by:
    \begin{equation*}
    \begin{split}
        (\sigma(D_0), \sigma(\zeta_p)) & = (D_0, \zeta_p^r) \\
        (\tau(D_0), \tau(\zeta_p)) & = (D_1, \zeta_p) \\
    \end{split}
    \end{equation*}

    For $\zeta_p^r$ a generator. $\sigma$ again conjugates $\sqrt{\Delta} \in F_p \setminus \Q$, so $\sigma(\zeta_p^i\beta) = \zeta_p^{ri} \beta_1$, and $\sigma(D_i) = D_{-ri}$. However, we see that $\zeta_p^{-r}$ is of order $\frac{p - 1}{2}$ because 
    \begin{equation*}
        \left(\frac{-r}{p}\right)
        = \left(\frac{-1}{p}\right)\left(\frac{r}{p}\right) = (-1)^{p(p - 1)/2}(-1) = 1
    \end{equation*} 
    
    So $\sigma$ will be two $\frac{p - 1}{2}$-cycles, the square of the $p - 1$-cycle generated by $\sqrt{-r}$. And again, $\tau(D_i) = D_{i + 1}$. The Galois group of $d_h$ therefore contains the elements $\tau=\tau_1\tau_2, \tau_3 \in S_p$, where
    $\tau_1, \tau_2$ are the $(p-1)/2$ cycles
\begin{center}
\begin{tabular}{c c c c c c c c}
$\tau_1$
&$=$&
$(1$ & 
$-r$ & $(-r)^2$&$\ldots$ & 
$(-r)^{(p-3)/2}$&
$)$
\\
$\tau_2$
&$=$&
$(-1$ & 
$r$ & $-(-r)^2$&$\ldots$ & 
$-(-r)^{(p-3)/2}$&
$)$
\\
\end{tabular}
\end{center}
and $\tau_3$ is the $p$ cycle
\begin{center}
\begin{tabular}{c c c c c c c c c}
$\tau_3$
&$=$&
$(0$ & 
$1$ & $2$&$\ldots$ & 
$p-1$&
$)$
\\
\end{tabular}
\end{center}
The elements $\tau$, $\tau_3$ generate the normal subgroup $C_p \rtimes C_{(p - 1)/2} \vartriangleleft \mathrm{Aff}(\mathbb{F}_p)$, whose size is $\frac{p(p - 1)}{2} = |\mathrm{Gal}(L / \mathbb{Q})|$. Thus, the Galois group of $d_h$ is $$\mathrm{Gal}(L / \mathbb{Q}) \simeq C_p \rtimes C_{(p - 1)/2} \vartriangleleft \mathrm{Aff}(\mathbb{F}_p)$$
    
Now $\sqrt{\Delta} \notin L$ and $L(\sqrt{\Delta}) = K$ as in Case 1, so $$\mathrm{Gal}(K / \mathbb{Q}) \simeq \mathrm{Gal}(L / \mathbb{Q}) \times C_2 \simeq (C_p \rtimes C_{(p - 1)/2}) \times C_2$$ \qedhere
\end{proof}

\begin{remark}
    The automorphisms $\sigma, \tau$ we choose in the proof of Theorem~\ref{galois-groups} are again as in Jones' \cite[Section 3, p. 6]{reciprocal}.
\end{remark}

\begin{remark}
    In the cases of $p = 2$ and $p = 3$ (the solutions to which we mention in Section 1), our methods of proof fail beginning at Lemma \ref{B_i-distinct}, where we use the fact $p - 1 \mid 2p$. The issue is essentially that $\zeta_2 = -1 \in \mathbb{Q}$ for the case of $p = 2$ and that if $\sqrt{\Delta} \in F_3$, then $\mathbb{Q}(\sqrt{\Delta}) = F_3$ in the case of $p = 3$. These issues both invalidate the distinctness of the $B_i$ and $D_i$, which is key to Lemma \ref{D_i-properties}, on which all of our subsequent work regarding $d_h$ and its Galois group is based.
\end{remark}

\begin{remark}
    From Corollary 1.2 of \cite{reciprocal}, there exist an infinite number of polynomials of the form of case 1 of Theorem \ref{galois-groups}. 
\end{remark}

\nocite{solvable-sextics}
\nocite{survey}

\printbibliography

\end{document}